\documentclass[reqno]{amsart}
\usepackage[utf8]{inputenc}
\usepackage[T1]{fontenc}
\usepackage{lmodern}
\usepackage{hyperref}
\usepackage{amsmath}
\usepackage{amsfonts}
\usepackage{amssymb}
\usepackage{cancel}
\usepackage{color}
\usepackage{lipsum}
\numberwithin{equation}{section}
\newtheorem{theorem}{Theorem}[section]
\newtheorem{proposition}[theorem]{Proposition}

\newtheorem{definition}{Definition}[section]

\newtheorem{lemma}{Lemma}[section]

\renewcommand{\>}{\rangle}
\begin{document}
\title{linear Elliptic equations with nonlinear boundary conditions under strong resonance conditions}

\author[Alzaki Fadlallah, Edcarlos D.da Silva\hfil ]
{Alzaki Fadlallah, Edcarlos D.da Silva }

\address{Alzaki.M.M. Fadlallah$^{1}$, Edcarlos D.da Silva$^{2}$ \hfill\break
$^{1}$Department of Mathematics, University of Alabama at
Birmingham.
Birmingham, Alabama 35294-1170, USA
 \hfill\break
$^{2}$Universidade Federal de Goiás, Instituto de Matemática e Estatística, 74001-970 Goiânia, GO, Brazil
 }

\email{ zakima99@uab.edu, edcarlos@ufg.br }

%\date

\begin{abstract}
In this work we establish existence and multiplicity of solutions for elliptic problem with nonlinear boundary conditions under
strong resonance conditions at infinity. The nonlinearity is resonance at infinity and the resonance phenomena occurs precisely in the
first Steklov eigenvalue problem. In all results we use Variational Methods, Critical Groups and the Morse Theory.

\end{abstract}

\maketitle
{\small{Key Words: Linear elliptic problems, Nonlinear boundary conditions, Steklov conditions at the boundary, Variational Methods, Morse Theory.}}

\section{Introduction}
The main goal of this paper is to prove the existence and multiplicity of solutions for the linear elliptic problems with  nonlinear boundary conditions

\begin{equation}\label{E1}
\begin{gathered}
 -\Delta u + c(x)u = 0 \quad\text{in } \Omega,\\
\frac{\partial u}{\partial \nu}=\mu_{1}u+f(x,u) \quad\text{on }
\partial\Omega,
\end{gathered}
\end{equation}
where $\Omega\subset\mathbb{R}^N,$  $N\geq 2$ is a bounded domain  with boundary
$\partial \Omega$ of class $C^2$, $c\in L^p(\Omega)$. Thorough this work we assume that $p\geq N$ and  $c\geq 0$ a.e.; on $\Omega$
with strict inequality on a set of positive measure. Here $\partial/\partial\nu :=\nu\cdot\nabla$
denotes the outward (unit) normal derivative on $\partial\Omega$ and $\mu_1$ is first positive eigenvalue of the Steklov problem
\begin{equation}\label{E2}
\begin{gathered}
 -\Delta u + c(x)u = 0 \quad\text{in } \Omega,\\
\frac{\partial u}{\partial \nu}=\mu u \quad\text{on }
\partial\Omega.
\end{gathered}
\end{equation}
The nonlinear term  $f\colon\partial\Omega\times\mathbb{R}\rightarrow\mathbb{R}$
satisfies the well known Carath\'{e}odory conditions, i.e.; we assume that
\begin{flushleft}
$(f_{0})$ The function $f\colon\partial\Omega\times\mathbb{R}\rightarrow\mathbb{R}$ satisfies
\end{flushleft}
\begin{description}
 \item[$i)$] $f(.,u)$ is measurable on $\partial\Omega$, for each $u\in\mathbb{R}.$
\item[$ii)$] $f(x,.)$ is continuous on $\mathbb{R},$ for $a.e.x\in \partial\Omega.$
\end{description}
Moreover, we shall consider $f$ subcritical, i.e.; we assume that
\begin{flushleft}
$(f_{1})$ The function $f\colon\partial\Omega\times\mathbb{R}\rightarrow\mathbb{R}$ satisfies
\end{flushleft}
\begin{equation*}
|f(x,u)|\leq C(1+|u|^{p-1}),~u\in\mathbb{R}, ~~x\in\partial\Omega,
\end{equation*}
where $1\leq p < \frac{N+2}{N-2}.$\\

The main objective in this work is to consider linear elliptic problems under resonance conditions at infinity on the boundary. In this way, we assume also that
\begin{equation}\label{E2.1}
\begin{gathered}
 \displaystyle\lim_{|u|\to\infty}\frac{f(x,u)}{u}=0,
\end{gathered}
\end{equation}
holds uniformly and almost everywhere in $x\in \partial\Omega.$ The limit just above says that problem \eqref{E1} presents the resonance phenomena
at the first positive eigenvalue problem given by \eqref{E2}.

Our approach in this work for the problem \eqref{E1} purely variational. Here we mention that find weak solutions of \eqref{E1} in $H^{1}(\Omega)$ is equivalent to finding critical points of the $C^{1}$ functional $J:H^{1}\to\mathbb{R}$ given by
\begin{equation}\label{E3}
\begin{gathered}
 J(u):=\frac{1}{2}\left[\int_{\Omega}|\nabla u|^{2}\,dx+\int_{\Omega}c(x)u^{2}\,dx-\int_{\partial\Omega}\mu_{1}u^{2}\,dx\right]
 -\int_{\partial\Omega}F(x,u)\,dx,
\end{gathered}
\end{equation}
where $F(x,u)=\int^{u}_{0}f(x,s)\,ds,~~x\in\partial\Omega,~~ u\in\mathbb{R}.$\\

We point out that problem \eqref{E1} presents the Steklov resonance phenomena at the first Steklov eigenvalue by assertion \eqref{E2.1}. These
problems have been studied by many authors in the recent years, we refer the reader to G.Auchmuty \cite{Auc12},
J. de Godoi, O. Miyagaki, R. Rodrigues \cite{GMR13}, da Silva \cite{daSi10}, A. Fadlallah \cite{Fad15}, N. Mavinga \cite{Mav12},
N. Mavinga, M. Nkashama \cite{MN10}, \cite{MN11}, \cite{MN12}, l. Steklov \cite{Ste1902}, and references therein.\\

Recall that the Steklov resonance phenomena for problem \eqref{E1} becomes stronger when the function $f$ is small at
infinity. In that case there is an interesting class of resonance problems called the Steklov strong resonance problems. Specifically, we
say that the problem \eqref{E1} presents the Steklov strong resonance phenomena when
\begin{equation*}(SSR)\label{E4}
\,\,\,\,\,\,\,\,\,
\begin{gathered}
\displaystyle\lim_{|u|\to\infty}f(x,u)=0,~~{\rm and}~~|F(x,u)|\leq\tilde{F}(x)\,\,a.e.;\,  {\rm in}\, x\in\partial\Omega,\,\,u\in\mathbb{R},
\end{gathered}
\end{equation*}
holds for some $\tilde{F}\in L^{q}(\partial\Omega),$ where $q\geq1.$ One more time
the limit just above is uniformly and almost everywhere in $x\in \partial\Omega.$\\
It is important to mention that well known nonquadraticity condition at infinity introduced in \cite{CM} was used in order 
to prove compactness conditions such as the Cerami condition which is essential in variational methods.\\
More precisely, the nonquadraticity condition at infinity can be written as
\begin{equation}\label{E5}
\begin{gathered}
(NQC)^{+}\,\,\,\,\,\,\lim_{|u|\to\infty}2F(x,u)-uf(x,u)=\infty,\\{\rm or}~\\
(NQC)^{-}\,\,\,\,\,\,\lim_{|u|\to\infty}2F(x,u)-uf(x,u)=-\infty,
\end{gathered}
\end{equation}
where the limit above are taken uniformly and almost everywhere in $x\in\partial\Omega.$
However, there are few results on the existence (see \cite{Fad15}) for problem \eqref{E1} when the conditions $(NQC)^{+}$ and $(NQC)^{-}$ are not
satisfied, i.e.; when $$|2F(x,u)-uf(x,u)|<\infty,$$ is bounded below and above. This case occurs in the Steklov strong resonance situation
when the function $f$ is small enough at infinity.
For instance, suppose $(SSR)$ and
\begin{equation}\label{E6}
\begin{gathered}
\displaystyle\limsup_{|u|\to\infty}|uf(x,u)| \leq C <\infty,
\end{gathered}
\end{equation}
holds for some $C > 0.$ Then the conditions $(NQC)^{+}$ and $(NQC)^{-}$ do not work in this case. In other words, under condition $(SSR)$ the function $f$ 
goes to zero faster than the function $h(u)=u$ goes to infinity proving that the nonquadraticity conditions $(NQC)^{+}$ and $(NQC)^{-}$ are not
verified.

To the best of our knowledge there is no results about the multiplicity for problem \eqref{E1} under strong resonant conditions at infinity. Here we give some existence and multiplicity solutions for the problem \eqref{E1} which complements and extend early results in the literature.\\

In this paper we establish the existence and multiplicity of solutions for \eqref{E1} assuming Steklov strong resonance
conditions at infinity such that the condition \eqref{E6} is verified. In other words, we wish to find existence and multiplicity
of solutions for \eqref{E1} where $f$ is small enough at infinity and the primitive of $f$ ($F$) is abounded function. \\

In this case we will introduce some conditions which are weaker than the conditions  $(NQC)^{+}$ or  $(NQC)^{-}$  in the Steklov strong resonance
situations. More specifically, motivated in part by \cite{daSi10}, we shall assume either\\
\begin{flushleft}
$(HOC)^{-}$ There is a function $a\in L^{1}(\partial\Omega)$ such that $a(x)\leq 0$ and
\end{flushleft}
\begin{equation}\label{E7}
\begin{gathered}
\lim_{|u|\to\infty}uf(x,u)\leq a(x).
\end{gathered}
\end{equation}
\begin{flushleft}
$(HOC)^{+}$ There is a function $b\in L^{1}(\partial\Omega)$ such that $b(x)\geq 0$ and
\end{flushleft}
\begin{equation}\label{E8}
\begin{gathered}
\lim_{|u|\to\infty}uf(x,u)\geq b(x),
\end{gathered}
\end{equation}
holds uniformly and a.e., in $x\in\partial\Omega.$ \\
We recall that $a(x)\leq 0$ a.e.; $x\in\partial\Omega$ with strict inequality on a set of positive
Lebesgue measure of $\partial\Omega.$ The $b(x)$ has a similar definition.\\
In order to control the resonance we will assume the conditions  $(HOC)^{-}$ or  $(HOC)^{-}$ proving that the functional $J$ satisfies the Cerami
condition {(We say  the functional $J$  satisfies Cerami condition at the level $c\in\mathbb{R}$, ($(Ce)_{c}$ in short ))
if any sequence $\{u_{n}\}_{n\in\mathbb{N}}\subset H^{1}(\Omega)$ such that
$$J(u_{n})\to c~~{\rm and }~~||J'(u_n)||(1+||u_{n}||)\to 0$$
as $n\to\infty$, possesses a convergent subsequence in $H^{1}(\Omega).$ Moreover, we say that $J$ satisfies $(Ce)$ condition when $(Ce)_{c}$ holds
for all $c\in\mathbb{R}.$ It is weaker than Palais Smale conditon}. Then, using variational methods we can prove some existence and multiplicity
results for problem \eqref{E1}.\\
Firstly, using Ekeland's Variational principle, we can prove the following existence result:
\begin{theorem}[Existence]\label{th1}
Suppose that $f$ satisfies $(f_0),\,(f_1)$ and $(SSR),$ $(HOC)^{+}$  or $(SSR),$ $(HOC)^{-}.$ Then, problem \eqref{E1} has at least one solution $u\in H^{1}(\Omega).$
 \end{theorem}

 In what follows we shall assume that
 $$f(x,0)=0 ~~{\rm in }\,\,\, \partial\Omega.$$
 Under this condition, we deduce that $u \equiv 0$ is a trivial solution of problem \eqref{E1}.  Hence the key point here is to ensure the
 existence of nontrivial solutions for problem \eqref{E1}.

Now we shall consider the following hypotheses:\\
 $(BH1)$ $\,\,\,\,$ 
 The function $f$ possesses the following growth at the origin:
 $$\limsup_{u\to 0}\frac{f(x,u)}{u}<0, $$
 uniformly and almost everywhere in $x\in \partial\Omega.$\\
 $(BH2)$ $\,\,\,\,$  There are real numbers $a^{-}<0<a^{+}$ such that
 $$\int_{\partial\Omega}F(x,a^{\pm}\varphi_{1})\,dx>0,$$
 where $\varphi_{1}$ is the first Steklov eigenfunction corresponding to the first eigenvalue Steklov problem \eqref{E2}.\\

 $(BH2)^{\prime}$ $\,\,\,\,$  Given $\mu_{2} > \mu_{1},$ where $\mu_{2}$ is  second eigenvalue for the Steklov problem \eqref{E2}  we have
$$F(x,u)\leq \frac{\mu_{2}-\mu_{1}}{2}|u|^{2}, ~~\forall~~ u \in \mathbb{R}, x \in \partial\Omega.$$

 In this way, combining the Ekeland's Variational methods and the Mountain Pass Theorem, we can show the following theorem:
 \begin{theorem}\label{th2}
Suppose that $f$ satisfies $(f_0),\,(f_1)$ and $(SSR),$ $(HOC)^{+}$  or $(SSR),$ $(HOC)^{-}.$ In addition, assume that $(BH1)$ and $(BH2)$ holds.
Then, problem \eqref{E1} has at least three nontrivial solutions  $u_{\pm},~~u_{1}.$
 \end{theorem}

Now we shall use the Saddle Point Theorem together with Ekelan's Variational Principle and Morse Theory proving the following result:
 \begin{theorem}\label{th3}
Suppose that $f$ satisfies $(f_0),\,(f_1)$ and $(SSR),$ $(HOC)^{+}$  or $(SSR),$ $(HOC)^{-}.$ In addition, assume that $(BH2)'$ holds.
Then, problem \eqref{E1} has at least three nontrivial solutions  $u_{\pm},~~u_{2}.$
 \end{theorem}

Moreover, we consider the following hypothesis:\\
$(BH3)$ $\,\,\,\,$ There are $r>0$ and $\epsilon\in(0,\mu_{2}-\mu_{1})$ such that
$$0\leq F(x,u)\leq\frac{\mu_{2}-\mu_{1}-\epsilon}{2}|u|^{2}, ~~\forall~~|u|\leq r.$$
Then, applying the Three Critical Point Theorem, we can also prove the following multiplicity result:
 \begin{theorem}[Multiplicity]\label{th4}
Suppose that $f$ satisfies $(f_0),\,(f_1)$ and $(SSR),$ $(HOC)^{+}$  or $(SSR),$ $(HOC)^{-}.$ In addition, assume that $(BH3)$ hold.
Then, problem \eqref{E1} has at least two nontrivial  solutions.
 \end{theorem}

 This work is organized as follows: In Section 2 we give the main properties for the eigenvalue Steklov problem and for the geometry of $J$. In section 3 we prove that $J$ satisfies the Cerami compactness condition. Section 4 is devoted to the proof of our main theorems.

 \section{Preliminares}
In this section we shall discuss the main properties of the eigenvalue Steklov problem. Later on, we shall also discuss the geometry of $J$ 
proving that $J$ admits the mountain pass geometry, saddle point geometry or a local linking geometry at the origin.\\
Let's define the real inner-product as
$$\<u,v\>_{c}:=\int_{\Omega}\triangledown u.\triangledown v+\int_{\Omega}c(x)uv~~~~\forall u,v\in  H^{1}(\Omega), $$
for the proof see \cite{Fad15}.
\begin{proposition}\label{Pro2.1}
Under the assumptions built on $c$ the norm
$$\|u\|_{c} = \sqrt{\int_{\Omega} |\nabla u|^{2} dx + \int_{\Omega} c(x) u^{2}dx},\,\, u \in H^{1}(\Omega),$$
defines a norm which is equivalent to the usual norm on $H^{1}(\Omega)$.
\end{proposition}
\begin{proof}[Proof of Proposition \ref{Pro2.1}] Clearly $||.||_{c}$ is a norm.
We show that $||.||_{c}$ is equivalent to the usual norm on $H^{1}(\Omega),$ recall that the norm on $H^{1}(\Omega)$
$$||u||^{2}:=||u||_{H^{1}(\Omega)}^{2}=\int_{\Omega}|\nabla u|^{2}\, dx+\int_{\Omega} u^{2}\,dx=||\nabla u||_{L^{2}(\Omega)}^{2}+||u||_{L^{2}(\Omega)}^{2}.$$
Since, we know that (from above) $\int_{\Omega} c(x) u^{2}dx<\infty,$  $\int_{\Omega} c(x) u^{2}dx\leq C||u||_{L^{2}(\Omega)}^{2},$ where $C$ is positive constant,
so we have that  $$\|u\|_{c}^{2}=||\nabla u||^{2} + \int_{\Omega} c(x) u^{2}dx\leq||\nabla u||_{L^{2}(\Omega)}^{2} +C||u||_{L^{2}(\Omega)}^{2}
\leq C_{1}||u||^{2}_{H^{1}(\Omega)},$$  so we have that $$\|u\|_{c}\leq C_{2}||u||_{H^{1}(\Omega)},$$ since $||u||_{c}$  (where $C_{2}$ is positive constant) is
continuous and quadratic on $H^{1}(\Omega).$ \\
Conversely,
there is an $\alpha >0$ such that
$$\|u\|_{c}^{2}\geq\alpha\int_{\Omega}u^2\,dx\,\forall\,u\in H^{1}(\Omega).$$ (For the proof see Theorem 3.2 in \cite{Auc12})
we have that
$$\alpha^{-1}||u||^{2}_{c}\geq||u||^{2}_{L^{2}(\Omega)},$$
$$||u||^{2}_{L^{2}(\Omega)}\leq\alpha^{-1}||u||^{2}_{c},$$
$$||u||^{2}_{L^{2}(\Omega)}+||\nabla u||_{L^{2}(\Omega)}^{2}\leq\alpha^{-1}||u||^{2}_{c}+||\nabla u||_{L^{2}(\Omega)}^{2}\leq\alpha^{-1}||u||^{2}_{c}+||u||^{2}_{c},$$
so
$$||u||^{2}_{H^{1}(\Omega)}\leq(1+\alpha^{-1})||u||^{2}_{c}.$$
Therefor,
$$||u||_{H^{1}(\Omega)}\leq\sqrt{(1+\alpha^{-1})}||u||_{c}.$$
Thus, the two norms are equivalent.
\end{proof}
Thorough this work we shall use the norm $\|.\|_{c}$ given above which it became $(H^{1}(\Omega), \|.\|_{c})$ an Hilbert space. The inner product is given by
\begin{equation}
\langle u , v\rangle_{c} = \int_{\Omega} \nabla u \nabla v dx + \int_{\Omega} c(x) u v dx, u , v \in H^{1}(\Omega).
\end{equation}
\begin{lemma}\label{Nem}
Suppose that $f$ satisfies $(f_0),\,(f_1)$ and there are constants $1\leq p,q<\infty$  and
$a,b>0$ such that for all $x\in\bar{\Omega}$ $u\in H^{1}(\Omega)$
\begin{equation}\label{Eqs4}
\begin{gathered}
 |f(x,u|\leq a+b|u|^{\alpha}~~{\rm with}~~\alpha=\frac{p}{q}.
\end{gathered}
\end{equation}
Define $$\mathbb{F}:L^{p}(\partial\Omega)\to L^{q}(\partial\Omega),$$ by
$$\mathbb{F}(\psi(x))=f(x,\psi(x))~~~\forall ~~\psi\in L^{p}.$$
Then, Nemytsk\^{\i}i operator $\mathbb{F}$ is a continuous map from $L^{p}(\partial\Omega)$ to $L^{q}(\partial\Omega).$
 \end{lemma}
 \begin{proof}[See \cite{Rab1986} For The Proof of Lemma \ref{Nem}]

 \end{proof}
Let $c \in L^{p}(\Omega),\,\, p \geq N,\,\, c \geq 0$ with strict inequality holding in some subset with positive Lebesgue measure.

 \begin{lemma}{\label{lmac}}
 Assume that $c$ satisfy the above condition. Then we have the following properties:
\begin{description}
 \item[i] For all $u\in H^{1}(\Omega) ,$
\begin{equation}\label{10}
\begin{gathered}
\mu_{1}||u||_{\partial}^{2}:=\mu_{1}\int_{\partial\Omega}u^2\leq\int_{\Omega}|\triangledown u|^2+\int_{\Omega}c(x)u^2=:||u||^{2}_{c},
\end{gathered}
\end{equation}
where $\mu_{1} >0$ is the least Steklov eigenvalue for equation \eqref{E2}. If equality holds in \eqref{10}, then $u$ is a multiple of
an eigenfunction of equation \eqref{E2} corresponding to $\mu_{1}$
 \item[ii] For every $v\in\oplus_{i\leq j}E(\mu_{i}),$ and $w\in\oplus_{i\geq j+1}E(\mu_{i}),$ we have that
\begin{equation}\label{11}
\begin{gathered}
||v||^{2}_{c}\leq\mu_{j}||v||^{2}_{\partial} ~~{\rm and}~~||w||^{2}_{c}\geq\mu_{j+1}||w||^{2}_{\partial},
\end{gathered}
\end{equation}
where $E(\mu_{i})$ is the $\mu_{i}$-eigenspace and $\oplus_{i\leq j}E(\mu_{i})$ is span of the eigenfunctions associated
to eigenvalues up to $\mu_{j}$
\end{description}
\end{lemma}
\begin{proof}[Proof of Lemma \ref{lmac}]
 If $u=0$, then the inequality \eqref{10} holds. otherwise, if  $0\neq u\in H^{1}(\Omega),$
then $u=u_{1}+u_{2},$ where $u_{1}\in[H^{1}_{0}(\Omega)]^{\bot} $, and $u_{2}\in H^{1}_{0}(\Omega).$ Therefore, by the $c-$orthogonality, and the characterization of
 $\mu_{1}$ (i.e.; $\mu_{1}||u_{1}||^{2}_{\partial}\leq||u_{1}||^{2}_{c})$ we get that
$$\mu_{1}||u||^{2}_{\partial}=\mu_{1}(||u_{1}||^{2}_{\partial}+||u_{2}||^{2}_{\partial}\leq||u_{1}||^{2}_{c}+||u_{2}||^{2}_{c}=||u||^{2}_{c}.$$
Thus, the inequality \eqref{10} holds.\\
Now assume we have that
$$||u||^{2}_{c}=\mu_{1}||u||^{2}_{\partial}\implies\mu_{1}= \frac{||u||^{2}_{c}}{||u||^{2}_{\partial}}.$$
we know that $\mu_{1}=\frac{||\varphi_{1}||^{2}_{c}}{||\varphi_{1}||^{2}_{\partial}},$ where $\varphi_{1}$ the eigenfunction corresponding to $\mu_{1}.$
Therefore, $u$ is a multiple of an eigenfunction of equation \eqref{E2} corresponding to $\mu_{1}$\\
We have that
$$||v||^{2}_{c}=\displaystyle\sum^{\infty}_{j=1}\mu_j|c_j|^2~~\forall ~~v\in\oplus_{i\leq j}E(\mu_{i}).$$
Now let  $\mu_{j}=\max\mu~\forall i\leq j,$ then we have that
$$||v||^{2}_{c}=\displaystyle\sum^{\infty}_{i=1}\mu_i|c_i|^2\leq\max\mu\displaystyle\sum^{\infty}_{i=1}|c_j|^2
=\mu_{j}||v||^{2}_{\partial}~\forall~v\in\oplus_{i\leq j}E(\mu_{i}).$$
$$||w||^{2}_{c}=\displaystyle\sum^{\infty}_{j=1}\mu_j|c_j|^2~~\forall\,\, w\in\oplus_{i\leq j}E(\mu_{i}).$$
Now let  $\mu_{j+1}=\min\mu~\forall\,\, i\geq j+1,$ then we have that
$$||w||^{2}_{c}=\displaystyle\sum^{\infty}_{j=1}\mu_j|c_j|^2\geq\min\mu\displaystyle\sum^{\infty}_{j=1}|c_j|^2
=\mu_{j+1}||w||^{2}_{\partial}~\forall~w\in\oplus_{i\geq j+1}E(\mu_{i}).$$
\end{proof}
\begin{lemma}\label{cont}
 $$J\in C^{1}(H^{1}(\Omega),\mathbb{R}),$$ and
 $$J'(u)v=\int_{\Omega}\nabla u\nabla v\,dx+\int_{\Omega} c(x)uv\,dx-\int_{\partial\Omega}\mu_{1}uv\,dx-\int_{\partial\Omega}f(x,u)v\,dx\,,\forall ~~v\in H^{1}(\Omega),$$
where $J'(u)$ denotes the Fr\'{e}chet derivative of $J$ at $u.$ Moreover,
$$J_{2}(u)=\int_{\partial\Omega}F(x,u)\,dx\,,$$
is weakly continuous, and $J_{2}'$ is compact.
 \end{lemma}
\begin{proof}[Proof of Lemma \ref{cont}]
 {Set $$J_{1}=\int_{\Omega}\nabla u\nabla v\,dx+\int_{\Omega} c(x)uv\,dx-\int_{\partial\Omega}\mu_{1}uv\,dx.$$
 Then, $J(u)=J_{1}(u)-J_{2}(u).$ It follows from assumption $iv:$, the Sobolev emdeding of $H^{1}(\Omega)$ into $L^{\frac{2N}{N-2}}(\Omega),$
 the continuity of the trace operator from $H^{1}(\Omega)$ into $L^{\frac{2(N-1)}{N-2}}$ and the H\"{o}lder inequality
 that $J$ and $J'$ are well defined. Using arguments similar to those in the proof of Proposition $B.10$ in \cite{Rab1986} one sees that
 $J_{2}$ belong to $C^{1}(H^{1}(\Omega),\mathbb{R})$ with Fr\'{e}chet derivative by the first two terms of $J'(u).$ We shall now prove that $J_{2}$
 also belong to $C^{1}(H^{1}(\Omega),\mathbb{R}),$ that it is weakly continuous and that $J_{2}'(u)$ is compact. \\
 We first prove that $J_{2}$ is Fr\'{e}chet differentiable on $H^{1}(\Omega),$ and that $J_{2}'(u)$ is continuous. For this purpose, let $u\in H^{1}(\Omega),$
 we claim that given $\epsilon>0,$ there exists $\delta=\delta(\epsilon,u)$ such that
  $$|J_{2}(u+v)-J_{2}(u)-J_{2}'(u)v|\leq\epsilon||v||_{c},$$ for all $v\in H^{1}(\Omega)$ with $||v||_{c}<\delta.$ Set
  $$\Psi:=|F(x,u+v)-F(x,u)-f(x,u)v|.$$
  It therefore follows that
  $$|J_{2}(u+v)-J_{2}(u)-J_{2}'(u)v|\leq\int_{\partial\Omega}\Psi.$$ Define
  $$S_{1}:=\{x\in \partial\Omega:|u(x)|\geq\vartheta\},$$
  $$S_{2}:=\{x\in \partial\Omega:|v(x)|\geq\kappa\},$$
  $$S_{2}:=\{x\in \partial\Omega:|u(x)|\geq\vartheta~\&~|v(x)|\geq\kappa\},$$
 where $\vartheta$ and $\kappa$ will be defined later. It then follows that
 $$\int_{\partial\Omega}\Psi\leq\sum_{i=1}^{3}\int_{S_{i}}\Psi.$$
 By the Mean Value Theorem we get that

 \begin{equation}\label{E8.01}
\begin{gathered}
F(x,\xi+\eta)-F(x,\xi)=f(x,\xi+\theta\eta)\eta,
\end{gathered}
\end{equation}
where $\theta\in(0,1).$ It follows from (\ref{E8.01}) and $iv:$ that
$$\int_{S_{1}}|F(x,u+v)-F(x,u)|\leq\int_{S_{1}}|f(x,u+\theta v)||v|$$
$$\leq\int_{S_{1}}[a_{1}+a_{2}|u+\theta v|^{s}]|v|\leq\int_{S_{1}}[a_{1}+a_{2}\left(|u|+|v|\right)^{s}]|v|,$$
where $a_1,~~a_2>0$ are constants, using H\"{o}lder inequality we obtain

$$\int_{S_{1}}|F(x,u+v)-F(x,u)|\leq a_{1}|S_{1}|^{\frac{N}{2N-2}}||v||_{L^{\frac{2N-2}{N-2}}(\partial\Omega)}+
a_{2}\left(\int_{S_{1}}|u|^{s}|v|+\int_{S_{1}}|v|^{s}|v|\right)$$
$$\leq\left[a_{1}|S_{1}|^{\frac{N}{2N-2}}+a_{2}|S_{1}|^{\frac{1}{\sigma}}\left(||u||^{s}_{L^{s+1}(\partial\Omega)}
+||v||^{2}_{L^{s+1}(\partial\Omega)}\right)\right]||v||_{L^{\frac{2N-2}{N-2}}(\partial\Omega)},$$
where
\begin{equation}\label{E8.02}
\begin{gathered}
\frac{1}{\sigma}+\frac{s}{s+1}+\frac{N-2}{2N-2}=1.
\end{gathered}
\end{equation}
Notice that $\frac{s}{s+1}+\frac{N-2}{2N-2}<1,$ so there exists a $\sigma>1$ such that (\ref{E8.02}) is satisfied. Using the continuity of the trace
operator from $H^{1}(\Omega)$ into $L^{t}(\partial\Omega)$ with $t\leq\frac{2(N-1)}{N-2},$ we obtain
\begin{equation}\label{E8.03}
\begin{gathered}
\int_{S_{1}}|F(x,u+v)-F(x,u)|\leq\left[a_1|S_{1}|^{\frac{N}{2N-2}}+a_{3}|S_{1}|^{\frac{1}{\sigma}}\left(||u||^{s}_{c}+||v||^{s}_{c}\right)\right]||v||_{c}.
\end{gathered}
\end{equation}
\begin{equation}\label{E8.04}
\begin{gathered}
\int_{S_{1}}|F(x,u+v)-F(x,u)|\leq a_4||v||_{c}\left[|S|_{1}^{\frac{N}{2N-2}}+|S_{1}|^{\frac{1}{\sigma}}\left(||u||^{s}_{c}+||v||^{s}_{c}\right)\right].
\end{gathered}
\end{equation}
Similarly
\begin{equation}\label{E8.05}
\begin{gathered}
\int_{S_{1}}|f(x,u)v|\leq a_5||v||_{c}\left[|S|_{1}^{\frac{N}{2N-2}}+|S_{1}|^{\frac{1}{\sigma}}\left(||u||^{s}_{c}+||v||^{s}_{c}\right)\right].
\end{gathered}
\end{equation}
By the continuity of the trace operator from $H^{1}(\Omega)$ into $L^{t}(\partial\Omega)$ with $t\leq\frac{2(N-1)}{N-2}$ and
H\"{o}lder inequality,
$$||u||_{c}\geq a_{6}||u||_{L^{2}({S_{1}})}\geq a_{6}\vartheta|S_{1}|^{\frac{1}{2}}.$$
Hence,
$$|S_{1}|^{\frac{1}{\sigma}}\leq\left(\frac{||u||_{c}}{a_{6}\vartheta}\right)^{\frac{2}{\sigma}}=:M_{1}$$ and
$$|S_{1}|^{\frac{N}{2(N-1)}}\leq\left(\frac{||u||_{c}}{a_{6}\vartheta}\right)^{\frac{N}{N-1}}=:M_{2},$$
$M_{1},\,M_{2}\to 0$ as $\vartheta\to\infty.$ Therefore,
$$\int_{S_1}\Psi\leq a_{7}\left[M_{2}+M_{1}\left(||u||^{s}_{c}+||v||^{s}_{c}\right)\right]||u||_{c}.$$
We can assume $\delta\leq 1$ and choose $\vartheta$ large such that
$$a_{7}\left[M_{2}+M_{1}\left(||u||^{s}_{c}+||v||^{s}_{c}\right)\right]\leq\frac{\epsilon}{3}.$$
Hence, $$\int_{S_{1}}\Psi\leq\frac{\epsilon}{3}||v||_{c}.$$
Similarly
$$\int_{S_{2}}\Psi\leq a_3\int_{S_2}[1+(|u|+|v|)^{2}]|v|$$
$$\leq a_{4}(1+||u|^{s}_{c}+||v||^{s}_{c})\left(\int_{S_{2}}|v|^{s+1}\left(\frac{|v|}{\kappa}\right)^{m-(s+1)}\right)^{\frac{1}{s+1}}\,\,{\rm with}\,
\,m=\frac{2(N-1)}{N-2}$$
$$\leq a_{6}\kappa^{\frac{s+1-m}{s+1}}(1+||u|^{s}_{c}+||v||^{s}_{c})||v||^{\frac{m}{s+1}}_{L^{m}(S_{2})}$$
$$\leq a_{6}\kappa^{\frac{s+1-m}{s+1}}(1+||u|^{s}_{c}+||v||^{s}_{c})||v||^{\frac{m}{s+1}}_{c}. $$
Since $F\in C^{1}(\bar{\Omega}\times\mathbb{R}),$ given any $\hat\epsilon,\,\hat\vartheta>0,$ there exists a
$\hat\kappa=\hat\kappa(\hat\epsilon,\hat\vartheta)$} such that
$$|F(x,\xi+h)-F(x,\xi)-f(x,\xi)h|\leq\hat\epsilon|h|$$
for all $x\in\partial\Omega,$ $|\xi|\leq\hat\vartheta,$ and $|h|\leq\hat\kappa.$ In particular if $\hat\vartheta=\vartheta$ and $\hat\kappa=\kappa,$
this implies
$$\int_{S_{3}}\Psi\leq\hat\epsilon\int_{S_{3}}|v|\leq a_{7}\epsilon||v||_{L^{1}(S_{3})}\leq a_{7}\epsilon||v||_{c}.$$
Choose $\hat\epsilon$ such that. This determines $\hat\kappa=\kappa.$ It follows
$$\int_{\partial\Omega}\Psi\leq\frac{2\epsilon}{3}||v||_{c}+a_{6}\kappa^{\frac{s+1-m}{s+1}}(1+||u||_{c}^{s}+||v||^{s}_{c})||v||_{c}^{\frac{m}{s+1}}.$$
Choose $\delta$ small so that $a_{6}\kappa^{\frac{s+1-m}{s+1}}(1+||u||_{c}^{s}+||v||^{s}_{c})\delta^{\frac{m}{s+1}})\leq\frac{\epsilon}{3}.$
Now, we shall prove that $J_{2}'(u)$ is continuous, let $u_{m}\to u$ in $H^{1}(\Omega),$ then by using H\"{o}lder inequality and the continuity of the trace
operator from $H^{1}(\Omega)$ into $L^{t}(\partial\Omega)$ with $t\leq \frac{2(N-1)}{N-2},$ we get
$$||J_{2}'(u_m)-J_{2}(u)||=\sup_{||v||_{c}\leq1}\left|\int_{\partial\Omega}f(x,u_{m})v-f(x,u)v\right|$$
$$\leq\sup_{||v||_{c}\leq1}\int_{\partial\Omega}|f(x,u_{m})v-f(x,u)||v|$$
$$\leq||f(.,u_m)-f(.,u)||_{L^{\frac{s+1}{s}}(\partial\Omega)}||v||_{L^{s+1}(\partial\Omega)}$$
$$\leq C||f(.,u_m)-f(.,u)||_{L^{\frac{s+1}{s}}(\partial\Omega)}.$$
By taking into account condition $iv:$ and Lemma \ref{Nem}, we see that the right-hand of the above inequality tends to zero  as
$m\to\infty.$ Hence, $J_{2}'$ is continuous. Now let us prove that $J_{2}$ is weakly continuous. Let $u_{n}\rightharpoonup u$ in $H^{1}(\Omega),$
it follows that $||u_{n}||_{c}< C.$ By the compactness of the trace operator, there exists a subsequence $u_{n_{k}}\to u$ in $L^{s+1}(\partial\Omega).$
$$|J_{2}(u_{n_{k}})-J_{2}(u)|\leq\int_{\partial\Omega}|f(x,\xi_{n_k})|u_{n_k}-u|~~\,{\rm by\, the \,Mean\, Value\, Theorem}$$
$$\leq||f(.,u_{n_k}||_{L^{\frac{s+1}{s}}(\partial\Omega)}||u_{n_k}-u||_{L^{s+1}(\partial\Omega)}~\,{\rm by\, H\"{o}lder \,inequality.}$$
Therefore, by Lemma \ref{Nem} we get that $J_{2}(u_{n_k})\to J_{2}(u).$ We claim that\\ $J_{2}(u_{n})\to J_{2}(u),$
hence $J_{2}(u_{n})\to J_{2}(u).$ Suppose by contradiction that $J_{2}(u_{n})\not\to J_{2}(u),$ then there exists a subsequence $\{u_{n_j}\}$ such that
$|J_{2}(u_{n_{j}})-J_{2}(u)|\geq\epsilon.$ But the sequence $\{u_{n_j}\}$ has a subsequence (we call again $\{u_{n_j}\}$) which convergent to $u$
in $L^{s+1}(\partial\Omega)$ and $J_{2}(u_{n_j})\to J_{2}(u).$ This leads to a contradiction. Thus, $J_{2}(u_{n})to J_{2}(u).$\\
Finally, let us prove that $J_{2}'$ is compact. Let $\{u_{n}\}$ be a bounded sequence in $H^{1}(\Omega)$, then there exists a subsequence
$u_{n_k}\rightharpoonup u$ in $H^{1}(\Omega).$ Therefore, $u_{n_{k}}\to u$ in $L^{s+1}(\partial\Omega).$ Then,
$$||J_{2}'(u_{n_k})-J_{2}'(u)||\leq C||f(.,u_{n_k})-f(.,u)||_{L^{\frac{s+1}{s}}}.$$
By Lemma \ref{Nem} we get that $J_{2}'(u_{n_k})\to J_{2}'(u).$ Thus, $J_{2}'$ is compact.
\end{proof}
\begin{definition}
 A weak solution of the equation \ref{E1}, we mean a function $u\in H^{1}(\Omega),$ such that
 \begin{equation}\label{E9}
\begin{gathered}
\int_{\Omega}\nabla u\nabla v\,dx+\int_{\Omega} c(x)uv\,dx-\int_{\partial\Omega}\mu_{1}uv\,dx-\int_{\partial\Omega}f(x,u)v\,dx\,=0,~~\forall v\in H^{1}(\Omega)
\end{gathered}
\end{equation}
\end{definition}
\begin{proposition}\label{auxiliar0}
Suppose $(SSR),$ $(BH1),$ $(BH2)$ hold. Then the functional $J$ admits the following Mountain Pass Geometry:
\begin{description}
  \item[$i)$] There exists $\rho > 0$ and $\alpha > 0$ such that
  $$J(u) \geq \rho \,\, \mbox{for any} \,\, u \in H^{1}(\Omega) \,\, \mbox{such that} \,\, \|u\|_{\partial} = \alpha.$$
  \item[$ii)$] There exists $e \in H^{1}(\Omega)$ such that $J(e) < 0$ and $\|e\|_{\partial} > \alpha$.
\end{description}
\end{proposition}
\begin{proof}[Proof of Proposition \ref{auxiliar0}]
First we show $i).$ From $(BH1)$ we have that
$$\limsup_{u\to 0} \frac{f(x,u)}{u} <0< \gamma.$$
This implies that
$$F(x,u)\leq \frac{\gamma}{2}|u|^{2}+C|u|^{q} \, u\in\mathbb{R},\,x\in\partial\Omega.$$
for some $q\in(2,\frac{N+2}{N-2})$ and $C>0$
where $0<\gamma<\mu_{1}.$ By  Sobolev's embedding and the trace theorem we have that
$$J(u)\geq\frac{1}{2}\big[\int_{\Omega} | \nabla u|^{2} \,dx+\int_{\Omega} c(x) u^{2}\,dx-\int_{\partial\Omega}\mu_{1} u^{2}\,dx\big]-
\frac{\gamma}{2} \int_{\partial\Omega} u^{2}\,dx- C \int_{\partial\Omega} u^{q}\,dx,$$
from Lemma \ref{lmac} we have that
$$\int_{\Omega} | \nabla u|^{2} \,dx+\int_{\Omega} c(x) u^{2}\,dx-\int_{\partial\Omega}\mu_{1} u^{2}\,dx \geq 0,$$ so we obtain
$$J(u)\geq - \frac{\gamma}{2} \int_{\partial\Omega} u^{2}\,dx- C \int_{\partial\Omega} u^{q}\,dx$$
$$=- \frac{\gamma}{2}||u||_{\partial}^{2}-C||u||^{q}_{\partial}$$
$$=\left(-\frac{\gamma}{2}-C||u||^{q-2}_{\partial}\right)||u||_{\partial}^{2}$$
for all $u \in H^{1}(\Omega).$ Putting $||u||_{\partial}=\alpha$ with $\alpha > 0$ small enough. The proof of $i)$ is complete.\\
Second we show $ii).$ Let $e=t\varphi_{1}$ where $\varphi_{1}$ the first eigenfunction for the Steklov problem with ($c\neq 0$),
and $t\in\mathbb{R}\setminus\{0\},$ for any $t$ big enough such that $||e||_{\partial}> \alpha,$ clearly $e \in H^{1}(\Omega).$\\
Now we show that $J(e)<0$
$$J(e)=J(t\varphi_{1})=\frac{t^{2}}{2}\left(\cancel{||\varphi_{1}||_{c}^{2}-\mu_{1}||\varphi_{1}||_{\partial}^{2}}^{0}\right)-\int_{\partial\Omega}F(x,t\varphi_{1})\,dx.$$
 So, we have that
 $$J(e)=-\int_{\partial\Omega}F(x,t\varphi_{1})\,dx.$$
 By $(BH2),$ we have that
  $$J(e)<0.$$ The proof of $ii)$ is complete.
\end{proof}

\begin{proposition}\label{auxiliar}
Suppose $(SSR)$. Then the functional $J$ is bounded from below. Moreover, the value
$$c_{inf} = \inf_{u \in H^{1}(\Omega)} J(u)$$
is a critical value of $J$, i.e.; there exists $u_{0} \in H^{1}(\Omega)$ such that $J(u_{0}) = c_{inf}$ and $J^{\prime}(u_{0}) \equiv 0$.
\end{proposition}
\begin{proof}
First of all, we shall prove that $J$ is bounded from below. The proof for this assertion follows arguing by contradiction. Consider a sequence $(u_{n}) \in H^{1}(\Omega)$ in such way that $J(u_{n}) \rightarrow - \infty$ as $n \rightarrow \infty$. Thus we obtain $\|u_{n}\| \rightarrow \infty$. Without any loss of generality we suppose that $J(u_{n}) \leq 0, n \in \mathbb{N}$. Under this condition we have
\begin{equation}\label{FF}
 J(u_{n}) = \int_{\Omega} |\nabla u_{n}|^{2} dx + \int_{\Omega} c(x)u^{2}_{n}dx - \mu_{1} \int_{\partial \Omega} u_{n}^{2}ds - \int_{\partial \Omega} F(x,u_{n}) ds \leq 0.
\end{equation}
Define $v_{n} = \frac{u_{n}}{\|u_{n}\|}$. Hence there exists $v \in H^{1}(\Omega)$ such that $v_{n} \rightharpoonup v$ in $H^{1}(\Omega)$. Dividing the inequality in \eqref{FF} we obtain
\begin{equation}\label{r1}
 \int_{\Omega} |\nabla v_{n}|^{2} dx + \int_{\Omega} c(x)v^{2}_{n}dx - \mu_{1} \int_{\partial \Omega} v_{n}^{2}ds - \int_{\partial \Omega} \dfrac{F(x,u_{n})}{u_{n}^{2}} v_{n}^{2}ds \leq 0.
\end{equation}
Using $(SSR)$ and Lebesgue Dominated Convergence Theorem we have
\begin{equation}\label{w4}
\lim_{n\rightarrow \infty} \int_{\partial \Omega} \dfrac{F(x,u_{n})}{u_{n}^{2}} v_{n}^{2}ds = 0.
\end{equation}
This together with \eqref{r1} imply that
\begin{equation}
 \lim_{n\rightarrow \infty} \left\{\int_{\Omega} |\nabla v_{n}|^{2} dx + \int_{\Omega} c(x)v^{2}_{n}dx - \mu_{1} \int_{\partial \Omega} v_{n}^{2}ds \right\} \leq 0.
\end{equation}
Thanks to variational inequality for $\mu_{1}$  we also that
\begin{equation}
 \lim_{n\rightarrow \infty} \left\{\int_{\Omega} |\nabla v_{n}|^{2} dx + \int_{\Omega} c(x)v^{2}_{n}dx - \mu_{1} \int_{\partial \Omega} v_{n}^{2}ds \right\} = 0.
\end{equation}.

Now using the compact Sobolev embedding we see that
\begin{equation}
\lim_{n\rightarrow \infty} \int_{\Omega} c(x)v^{2}_{n}dx =   \int_{\Omega} c(x)v^{2}dx,  \lim_{n\rightarrow \infty} \int_{\partial \Omega} v_{n}^{2}ds =
\int_{\partial \Omega} v^{2}ds.
\end{equation}
Furthermore, using the fact that the norm is weakly lower semicontinuous we have that
\begin{eqnarray}
\|v\|^{2} \leq \liminf_{n \rightarrow \infty} \|v_{n}\|^{2} &=& \liminf_{n \rightarrow \infty} \int_{\Omega} |\nabla v_{n}|^{2} dx + \int_{\Omega} c(x)v^{2}_{n}dx \nonumber \\
&=& \lim_{n \rightarrow \infty} \mu_{1} \int_{\partial \Omega} v_{n}^{2}ds = \mu_{1} \int_{\partial \Omega} v^{2}ds \nonumber \\
&\leq&  \int_{\Omega} |\nabla v|^{2} dx + \int_{\Omega} c(x)v^{2}dx = \|v\|^{2}.
\end{eqnarray}
In particular, we see that
\begin{equation}
\|v\|^{2} = \limsup_{n \rightarrow \infty} \|v_{n}\|^{2}
\end{equation}
and $v_{n} \rightarrow v$ in $H^{1}(\Omega)$. So $v = t \phi_{1}$ for some $t \in \mathbb{R}$. After that, using that $(v_{n})$ is normalized, the inequality \eqref{r1} says also that
\begin{equation}
\mu_{1} \int_{\partial \Omega} v_{n}^{2} \geq 1 + o_{n}(1), \, n \in \mathbb{N}.
\end{equation}
This ensures that $ \displaystyle \int_{\partial \Omega} v^{2} \geq \frac{1}{\mu_{1}} > 0$. Thus we have been showed that $v \neq 0$.

Now we shall write $u_{n} = t_{n} \phi_{1} + w_{n}$ where $t_{n} \in \mathbb{R}, w_{n} \in \bigoplus_{j = 2}^{\infty} E(\mu_{j})$. According to \eqref{r1}
and variational inequality for $\mu_{2}$ it follows that
\begin{eqnarray}
0 \leq \left( 1 - \dfrac{\mu_{1}}{\mu_{2}}\right) \|w_{n}||^{2} &\leq& \|w_{n}\|^{2} - \mu_{1} \int_{\partial \Omega} w_{n}^{2} ds \nonumber \\
&=& \|u_{n}\|^{2} - \mu_{1} \int_{\partial \Omega} u_{n}^{2} ds \leq \int_{\partial \Omega} \dfrac{F(x,u_{n})}{u_{n}^{2}} v_{n}^{2} ds. \nonumber \\
\end{eqnarray}
Using one more time \eqref{w4} we easily see that
$\|w_{n}|| \rightarrow 0$ as $n \rightarrow \infty$. As a consequence
\begin{equation}
\lim_{n \rightarrow \infty} \int_{\Omega} |\nabla u_{n}|^{2} dx + \int_{\Omega} c(x)u^{2}_{n}dx - \mu_{1} \int_{\partial \Omega} u_{n}^{2}ds = 0.
\end{equation}
In addition,  we easily see that
\begin{equation}\label{w5}
\lim_{n \rightarrow \infty} \int_{\partial \Omega} F(x, u_{n}) dx = +\infty.
\end{equation}

Now we mention that $(SSR)$ implies that
\begin{equation}
\left|F(x, u_{n}) \right| \leq  \tilde{F}(x), x \in \partial \Omega.
\end{equation}
In particular, we have
\begin{equation}
\lim_{n \rightarrow \infty } \left|\int_{\partial \Omega} F(x, u_{n}) dx \right| \leq \int_{\partial \Omega} \tilde{F}(x) dx < \infty.
\end{equation}
Thus we have a contradiction with \eqref{w5} proving that $J$ is bounded from below.

Later on, we shall prove that $J$ satisfies the Cerami condition, see Proposition \ref{cerami}. Since $J$ is bounded from below is standard from Ekeland's Variational Principe to show that $c_{\inf}$ is a critical value. Hence there exists $u_{0} \in H^{1}(\Omega)$ such that $J^{\prime}(u_{0}) \equiv 0$ and $J(u_{0}) = c_{\inf}$. This completes the proof.
\end{proof}

Now we define the sets
\begin{equation}
A^{+} = \left\{ t \phi_{1} + w : t \geq 0, w \in \bigoplus_{j =2}^{\infty} E(\mu_{j}) \right\}
\end{equation}
and
\begin{equation}
A^{-} = \left\{ t \phi_{1} + w : t \leq 0, w \in \bigoplus_{j =2}^{\infty} E(\mu_{j}) \right\}.
\end{equation}
The sets $A^{+}, A^{-}$ are nonempty and $A^{+} \cap A^{-} = \bigoplus_{j =2}^{\infty} E(\mu_{j})$. Now we shall minimize the functionals $J|_{A^{+}}$ and $J|_{A^{-}}$ proving that $J$ admits two distinct critical points.

\begin{proposition}\label{auxiliar1}
Suppose $(SSR)$ and $(BH1),(BH2), (BH2)^{\prime}$. Assume also $(HOC)^{-}$. Then the functional $J$ is bounded from below over the sets $A^{+}$ and $A^{-}$. Furthermore, the values
$$c^{+} = \inf_{u \in A^{+}} J(u), c^{-} = \inf_{u \in A^{-}} J(u) $$
are two critical values of $J$. Hence the functional $J$ admits two distinct critical points.
\end{proposition}
\begin{proof}
We show that $J$ is bounded from below over  $A^{+},$ similar proof for bounded from below  over $A^{-}$
we use the $\phi_{1}$ orthogonal to $w$ under $c-$norm and $\partial-$norm. Let $u\in A^{+}$
$$J(u)=J(t\phi_{1}+w)=\frac{t^{2}}{2}||\phi_{1}||^{2}_{c}+\frac{1}{2}||w||^{2}_{c}-\frac{t^{2}}{2}\mu_{1}||\phi_{1}||_{\partial}^{2}-
\frac{1}{2}\mu_{1}||w||_{\partial}^{2}-\int_{A^{+}}F(x,t\phi_{1}+w)$$
$$\geq \frac{1}{2}||w||^{2}_{c}-\frac{1}{2}\mu_{1}||w||_{\partial}^{2}-\int_{A^{+}}\frac{\mu_{2}-\mu_{1}-\epsilon}{2}|t\phi_{1}-w|^{2}$$
$$\geq -\frac{1}{2}\mu_{1}||w||_{\partial}^{2}-(\frac{\mu_{2}-\mu_{1}-\epsilon}{2})\left(t^{2}||\phi||^{2}_{c}+||w||^{2}_{c}\right)>C\,,$$
where $C$ is some constant,
so that
$J$ is bounded  from below over  $A^{+}$, and Similarly for $A^{-}$. \\
We show that $J$  has  critical value  on $A^{+}.$ Since $J$ is bounded from below over $A^{+}$ ant it is of $C^{1}$ class and satisfy $(PS)$ condition it follows from Theorem 4.15 \cite{Fad15} that
$J$ has critical value on $A^{+}.$ Similarly for $A^{-}.$  Define
$$c^{+} = \inf_{u \in A^{+}} J(u),\,\, c^{-} = \inf_{u \in A^{-}} J(u). $$
Now we show that $$c^{+}\neq c^{-}.$$
We consider the functionals $J^{\pm}=J|_{A^{\pm}}$. Since we obtain two critical points which are denote
by $u_{0}^{+}$ and $u_{0}^{-},$ respectively.
Thus, we have the following
$$c^{+}=J^{+}(u_{0}^{+}) = \inf_{u \in A^{+}} J(u),\,\, c^{-}=J^{-}(u_{0}^{-}) = \inf_{u \in A^{-}} J(u).$$
Moreover, we see that  $u_{0}^{+}$ and $u_{0}^{-}$ are nonzero critical points. By using $(BH2)$ we have that
\begin{equation}\label{En}
 J^{\pm}(u^{\pm}_{0}) \leq J(t^{\pm}\phi_{1})=-\int_{\partial\Omega}F(x,t^{\pm}\phi_{1})\leq 0.
\end{equation}

By using $(BH2)$, $(BH2)^{\prime}$, Lemma \ref{lmac} we deduce that $J$ restrict to $\bigoplus_{j =2}^{\infty} E(\mu_{j})$ is nonnegative. In fact, taking $w\in\bigoplus_{j =2}^{\infty} E(\mu_{j})$ we have the following estimates:
\begin{eqnarray}\label{Ew}
 J(w)&=&\frac{1}{2}\left(||w||^{2}_{c}-\mu_{1}||w||^{2}_{\partial}\right)-\int_{\partial\Omega}F(x,w)\,dx \nonumber \\
 &\geq&\frac{1}{2}(\mu_{2}-\mu_{1})||w||^{2}_{\partial}-\left(\frac{\mu_{2}-\mu_{1}}{2}\right)||w||^{2}_{\partial} =0. \nonumber \\
\end{eqnarray}
We show that $u_{0}^{+}$ and $u_{0}^{-}$ are distinct. The proof follows arguing by contradiction. Suppose that
$u_{0}^{+}=u_{0}^{-} = w \in \bigoplus_{j =2}^{\infty} E(\mu_{j}).$ Then, by the estimate in \eqref{Ew}
we have that $J(u_{0}^{\pm})<0\leq J(u_{0}^{\pm}).$ Therefore, we have contradiction. So, $u_{0}^{+}\neq u_{0}^{-}.$
Thus, $u_{0}^{\pm}$ are two distinct critical points of $J.$ The functional $J$ admits two distinct critical points. The proof is now complete.

\end{proof}

\begin{proposition}\label{auxiliar2}
Suppose $(SSR)$. Then the functional $J$ admits the following saddle point geometry
\begin{description}
  \item[$i)$] $J(u_{n}) \rightarrow + \infty, \|u\| \rightarrow \infty$ where $u \in \bigoplus_{j = 2}^{\infty} E(\mu_{j})$,
  \item[$ii)$] There exists $C > 0$ such that
  $$J(u) \leq C \,\, \mbox{for any}  \,\, u \in E(\mu_{1}).$$
\end{description}
\end{proposition}
\begin{proof}
According $(SSR)$ is quite standard to ensure that $i)$ is verified. We will omit the details in this case.

Now we shall prove the item $ii)$.  The proof follows by contradiction. Let $(u_{n}) \in E(\mu_{1})$ be an unbounded sequence such that
$J(u_{n}) \rightarrow \infty$ as $n \rightarrow \infty$. Clearly, using the fact that $E(\mu_{1})$ is unidimensional, we can rewrite $u_{n} = t_{n} \phi_{1}$ for some sequence $(t_{n}) \in \mathbb{R}$ such that $|t_{n}| \rightarrow \infty$ as $n \rightarrow \infty$. In this way, we obtain
\begin{equation}
J(u_{n}) = J(t_{n} \phi) = - \int_{\partial \Omega} F(x, t_{n} \phi_{1}) ds
\end{equation}
holds for any $n \in \mathbb{N}$ big enough. This identity implies that
\begin{equation}\label{ww1}
\lim_{n \rightarrow \infty} \int_{\partial \Omega} F(x, t_{n} \phi_{1}) ds = - \infty.
\end{equation}

Now we mention that $(SSR)$ implies that
\begin{equation}\label{w3}
|F(x, t_{n} \phi_{1})| \leq \tilde{F}(x), x \in \partial \Omega.
\end{equation}
In particular, the last assertion yields
\begin{equation}\label{w3}
\lim_{n \rightarrow \infty} \left|\int_{\partial \Omega} F(x, t_{n} \phi_{1}) ds \right| \leq \int_{\partial \Omega} \tilde{F}(x) dx < \infty.
\end{equation}
This is a contradiction with \eqref{ww1} proving that $J$ is bounded from above on $E(\mu_{1})$. So we finish the proof.
\end{proof}

\begin{proposition}\label{auxiliar3}
Suppose $(SSR)$. Then the functional $J$ admits the following Local Linking geometry: There exists $\delta > 0$
such that
\begin{description}
  \item[$i)$] $J(u) \geq 0$ for any $ \|u\| \leq  \delta$ where $u \in \bigoplus_{j = 2}^{\infty} E(\mu_{j})$,
  \item[$ii)$] $J(u) \leq 0$ for any $ \|u\| \leq  \delta$ where $u \in E(\mu_{1})$.
\end{description}
\end{proposition}
\begin{proof}
First we shall consider the proof for the item $i)$.  Let $w \in \bigoplus_{j = 2}^{\infty} E(\mu_{j})$ be a fixed function. Using $(BH3)$ it follows that
\begin{equation}
F(x, t) \leq \dfrac{\mu_{2} - \mu_{1} - \epsilon}{2} t^{2} + C |t|^{q}, x \in \partial \Omega, t \in \mathbb{R}
\end{equation}
where we put $q \in (2, 2_{\star})$. Hence the last estimate and Sobolev compact embedding imply that
\begin{equation}\label{w1}
J(w) \geq \int_{\Omega} |\nabla u|^{2} dx - \int_{\Omega} c(x) w^{2} dx - \mu_{1} \int_{\partial \Omega} w^{2} dx - \dfrac{\mu_{2} - \mu_{1} - \epsilon}{2} \int_{\partial \Omega} w^{2} ds - C \|w\|^{q}
\end{equation}
Thus the variational inequality for $\mu_{2}$ and \eqref{w1} provide us the following estimates
\begin{equation}
J(w) \geq \left( 1 - \dfrac{\mu_{2} - \epsilon}{\mu_{2}} \right)\|w\|^{2} - C \|w\|^{q} = \left\{ \left( 1 - \dfrac{\mu_{2} - \epsilon}{\mu_{2}} \right) - C \|w\|^{q -2} \right\} \|w\|^{2}
\end{equation}
As a consequence we obtain a number $\delta_{1} > 0$ in such way that
\begin{equation}
J(w) \geq \dfrac{\epsilon}{2\mu_{2}} \|w\|^{2} \geq 0, \|w\| \leq \delta_{1}, w \in \bigoplus_{j = 2}^{\infty} E(\mu_{j}).
\end{equation}
So we end the proof of item $i)$.

Now we shall consider the proof for the item $ii)$. Here we mention that any norms in $E(\mu_{1})$ are equivalents. Thus there exists $C > 0$ such that
\begin{equation}
\|u\|_{\infty} \leq C \|u\|, u \in E(\mu_{1}).
\end{equation}
In particular, putting $u \in E(\mu_{1})$ in such way that $\|u\| \leq \frac{r}{C}$ we obtain
$\|u\|_{\infty} \leq r$ where $r > 0$ is given by $(BH3)$. Define $\delta_{2} = \dfrac{r}{C}$.
Using one more time $(BH3)$ we also see that
\begin{equation}
J(u) = - \int_{\partial \Omega} F(x, u) ds \leq 0, \|u\| \leq \delta_{2}, u \in E(\mu_{1}).
\end{equation}
This fact proves the item $ii)$. Hence the proof of this proposition is achieved taking $\delta := \min(\delta_{1}, \delta_{2})$. So we finish the proof.
\end{proof}

\section{The proof of Cerami condition}
In this section we shall prove that $J$ satisfies the Cerami condition. As a first step we shall prove that any Cerami sequence for $J$
is bounded in $H^{1}(\Omega)$.

\begin{proposition}\label{bound}
Suppose $(SSR)$. Assume either $(HOC)^{-}$ or $(HOC)^{+}$ holds.  Then any Cerami sequence for the functional $J$ is bounded in $H^{1}(\Omega)$.
\end{proposition}

\begin{proof}
The proof of this proposition follows arguing by contradiction.
Let $u_{n} \in H^{1}(\Omega)$ be an unbounded Cerami sequence.  Define the function $v_{n} = \frac{u_{n}}{\|u_{n}\|}$. Hence $(v_{n})$ is bounded and there exists $v \in H^{1}(\Omega)$ in such way that $v_{n} \rightharpoonup v$ in $H^{1}(\Omega)$. The Sobolev compact embedding says also that
$v_{n} \rightarrow v \in L^{q}(\Omega), q \in [1, 2^{\star})$ ($2^{\star}=\frac{2N}{N-2}$) and $v_{n} \rightarrow v $ 
a. e.; in $\Omega$ such that $|v_{n}| \leq h$ for some  $h \in L^{q}(\Omega)$. 
Similarly, the compact embedings $H^{1}(\Omega) \subset L^{r}(\partial \Omega)$ 
imply that $v_{n} \rightarrow v$ in $L^{r}(\partial \Omega)$ and $v_{n} \rightarrow v$ a. e in $\partial \Omega$ for any $r \in [1, 2_{\star})
\,\,(2_{\star}=\frac{2N-1}{N-2})$. 

Now using that $(u_{n})$ is a Cerami sequence we have that
\begin{equation}
\int_{\Omega} \nabla u_{n} \nabla \phi dx + \int_{\Omega} c(x) u_{n} \phi dx  - \mu_{1} \int_{\partial \Omega} u_{n} \phi- \int_{\partial \Omega} f(x,u_{n}) \phi ds = \langle J^{\prime}(u_{n}), \phi \rangle, \phi \in H^{1}(\Omega).
\end{equation}
Dividing the last expression by $\|u_{n}\|$ we obtain
\begin{equation}\label{e0}
\int_{\Omega} \nabla v_{n} \nabla \phi dx + \int_{\Omega} c(x) v_{n} \phi dx  - \mu_{1} \int_{\partial \Omega} v_{n} \phi- \int_{\partial \Omega} \dfrac{f(x,u_{n})}{u_{n}} v_{n} \phi ds = o_{n}(1), \phi \in H^{1}(\Omega).
\end{equation}
According to (SSR) and the Dominated Convergence Theorem we see that
\begin{equation}\label{e1}
\lim_{n \rightarrow \infty}\int_{\partial \Omega} \dfrac{f(x,u_{n})}{u_{n}} v_{n} \phi ds = 0.
\end{equation}
Using \eqref{e0}, \eqref{e1} we conclude that
\begin{equation}
\int_{\Omega} \nabla v \nabla \phi dx + \int_{\Omega} c(x) v \phi dx  - \mu_{1} \int_{\partial \Omega} v \phi = 0, \phi \in H^{1}(\Omega).
\end{equation}
Moreover, using $\phi = u_{n}$ as testing function, we have that
\begin{eqnarray}\label{iden}
  \|v\|^{2}&=& \int_{\Omega} |\nabla v|^{2} dx + \int_{\Omega} c(x) v^{2} dx = \mu_{1} \int_{\partial \Omega} v^{2} ds = \mu_{1} \lim_{n \rightarrow \infty} \int_{\Omega} v_{n}^{2} ds \nonumber \\
   &=&  \lim_{n \rightarrow \infty} \left\{ -\frac{\langle J^{\prime} (u_{n}), u_{n}\rangle}{\|u_{n}\|^{2}} + \int_{\Omega} |\nabla v_{n}|^{2} dx + \int_{\Omega} c(x) v_{n}^{2} dx - \int_{\partial \Omega} \dfrac{f(x,u_{n})}{u_{n}} v_{n} ds \right\} \nonumber \\
   &=&  \lim_{n \rightarrow \infty} \left\{ -\frac{\langle J^{\prime} (u_{n}), u_{n}\rangle}{\|u_{n}\|^{2}} + 1 - \int_{\partial \Omega} \dfrac{f(x,u_{n})}{u_{n}} v_{n} ds \right\} = 1. \nonumber \\
\end{eqnarray}
Here was used the fact that $(v_{n})$ is normalized. In particular, using \eqref{iden} the weak convergence implies that $v_{n} \rightarrow v$ in $H^{1}(\Omega)$.  Thus we have been showed that $v$ is a nonzero weak solution for the eigenvalue problem
\begin{equation}\label{eig}
\begin{gathered}
 -\Delta u + c(x)u = 0 \quad\text{in } \Omega,\\
\frac{\partial u}{\partial \nu}=\mu_{1}u \quad\text{on }
\partial\Omega.
\end{gathered}
\end{equation}
As a consequence $v = t \phi_{1}$ for some $t \in \mathbb{R}\backslash \{ 0\}$ where $\phi_{1}$ denotes the first eigenvalue for the problem \eqref{eig}.

Now we observe that $|u_{n}| \rightarrow \infty$ on the set $[v \neq 0] := \{ x \in \Omega : |v(x)| \neq 0 \}$. Putting $\phi = u_{n}$ as testing function it follows that
\begin{equation}\label{ee0}
\int_{\Omega} |\nabla u_{n}|^{2} dx + \int_{\Omega} c(x) u_{n}^{2} dx  - \mu_{1} \int_{\partial \Omega} u_{n}^{2}- \int_{\partial \Omega} f(x,u_{n})u_{n}ds = \langle J^{\prime}(u_{n}), u_{n} \rangle.
\end{equation}

Now we shall write $u_{n} = t_{n}\phi_{1} + w_{n}$ where $t_{n} \in \mathbb{R}$ and $(w_{n}) \in \bigoplus_{j = 2}^{\infty} E(\mu_{j}).$ The main feature here is to prove that $(w_{n})$ is bounded sequence. In order to do that we take $v = w_{n}$ as testing function proving that
\begin{eqnarray}
0 \leq \left( 1 - \dfrac{\mu_{1}}{\mu_{2}}\right)\|w_{n}\|^{2} &\leq& \int_{\Omega} |\nabla w_{n}|^{2} dx - \int_{\Omega} c(x) w_{n}^{2} -\mu_{1} \int_{\partial \Omega} w_{n}^{2}  \nonumber \\ &\leq& \int_{\partial \Omega} f(x,u_{n}) w_{n} dx + o_{n}(1). \nonumber \\
\end{eqnarray}
Now using $(SSR)$ and Sobolev embedding $H^{1}(\Omega) \subset L^{2}( \partial \Omega)$ one finds
\begin{equation}
0 \leq \left( 1 - \dfrac{\mu_{1}}{\mu_{2}}\right)\|w_{n}\|^{2} \leq C + C \|w_{n}\|.
\end{equation}
As a consequence $(w_{n})$ is now bounded in $H^{1}(\Omega)$. Taking into account $(SSR)$ again and using $\phi = w_{n}$ as testing function we obtain
\begin{eqnarray}\label{e3}
0 \leq \left( 1 - \dfrac{\mu_{1}}{\mu_{2}}\right)\|w_{n}\|^{2} &\leq& \int_{\partial \Omega} |f(x, u_{n}) w_{n}| dx
 \nonumber \\
 &\leq& \|f(. , u_{n})\|_{L^{2}(\partial \Omega)} \|w_{n}\|_{L^{2}(\partial \Omega)} \leq C \|f(. , u_{n})\|_{L^{2}(\partial \Omega)} \|w_{n}\|. \nonumber \\
\end{eqnarray}
Therefore, using \eqref{e3} and Lebesgue Dominated Convergence Theorem, we see that $\|w_{n}\|\rightarrow 0$ as $n \rightarrow \infty$. This implies that
\begin{equation}\label{e4}
\int_{\Omega} |\nabla u_{n}|^{2} dx + \int_{\Omega} c(x) u_{n}^{2} dx  - \mu_{1} \int_{\partial \Omega} u_{n}^{2} = o_{n}(1).
\end{equation}
So, that \eqref{ee0} and \eqref{e4} give us
\begin{equation}\label{d0}
\lim_{n \rightarrow \infty } \int_{\partial \Omega} f(x,u_{n})u_{n}ds = 0.
\end{equation}
On the other hand, using Fatou's Lemma, we easily see that
\begin{equation}
\limsup_{n \rightarrow \infty } \int_{\partial \Omega} f(x,u_{n})u_{n}ds \neq 0.
\end{equation}
In fact, using assumption $(HOC)^{-}$ is not hard to see that
\begin{equation}\label{d1}
\limsup_{n \rightarrow \infty } \int_{\partial \Omega} f(x,u_{n})u_{n}ds \leq \int_{\partial \Omega} \limsup_{n \rightarrow \infty } f(x,u_{n})u_{n}ds \leq \int_{\partial \Omega} a(x) dx < 0.
\end{equation}
Similarly, using the assumption $(HOC)^{+}$ we can prove that
\begin{equation}\label{d2}
\liminf_{n \rightarrow \infty } \int_{\partial \Omega} f(x,u_{n})u_{n}ds \geq \int_{\partial \Omega} \liminf_{n \rightarrow \infty } f(x,u_{n})u_{n}ds \geq \int_{\partial \Omega} b(x) dx > 0.
\end{equation}
In conclusion, the equations \eqref{d1} and \eqref{d2} provide us a contradiction with \eqref{d0}. Thus the sequence $(u_{n})$ is now bounded. This finishes the proof.
\end{proof}

Now we stay in position to prove that any Cerami sequences for $J$ admits a subsequence which is strongly convergent sequence in $H^{1}(\Omega)$.
\begin{proposition}\label{cerami}
Suppose $(SSR)$. Assume also either $(HOC)^{-}$ or $(HOC)^{+}$ holds. Then the functional $J$ satisfies the $(Ce)_{c}$ condition at any level $c \in \mathbb{R}$.
\end{proposition}
\begin{proof}
Let $(u_{n})$ be a Cerami sequence for the functional $J$. According to Proposition \ref{bound} the sequence $(u_{n})$ is bounded. Hence $u_{n}\rightharpoonup u$ for some $u \in H^{1}(\Omega)$. Using the fact that $f$ is subcritical we can also prove that $(u_{n})$ strongly converges
in $H^{1}(\Omega)$. We omit the details.
\end{proof}

\section{The proof of the main theorems}

\begin{proof}[Proof of Theorem \ref{th1}]
 Initially we observe that $J$ is bounded from below, see Proposition \ref{auxiliar}. After that, the functional $J$ satisfies the Cerami condition, see Proposition \ref{cerami}. So that, using Ekeland's Variational Principle, we obtain a critical point $u_{0} \in H^{1}(\Omega)$ such that $J(u_{0}) = c_{\inf}$ where $c_{\inf} = \displaystyle \inf_{u \in H^{1}(\Omega)} J(u)$, see Proposition \ref{auxiliar}.
\end{proof}

\begin{proof}[Proof of Theorem \ref{th2}]
The main idea here is to minimize $J$ over the sets $A^{+}$ and $A^{-}$. Here we mention that $J$ is bounded from below, see Proposition \ref{auxiliar}. Under this condition we consider the functionals  $J|_{A^{+}}$ and $J|_{A^{-}}$. Minimizing these functionals we obtain two distinct critical points for $J$, see Proposition \ref{auxiliar1}. More precisely, we obtain two critical points $u_{+} \in A^{+} , u_{-}  \in A^{-}$ in such way that $J(u_{+}) < 0, J(u_{-}) < 0$ and
\begin{equation}
J(u_{+}) = \inf_{u \in A^{+}} J(u), J(u_{-}) = \inf_{u \in A^{-}} J(u).
\end{equation}

On the other hand, the functional admits the mountain pass geometry, see Proposition \ref{auxiliar0}. As $J$ satisfies the Cerami condition we obtain a critical point $u_{1}$ of mountain pass type. In particular, we have $J(u_{1}) > 0$ proving that $u_{1}, u_{+}, u_{-}$ are distinct critical points for $J$. Hence the problem \eqref{E1} admits at least three nontrivial solutions. This completes the proof.
\end{proof}

\begin{proof}[Proof of Theorem \ref{th3}]
Now we shall consider the saddle point geometry given in Proposition \ref{auxiliar2}. Taking into account that $J$ satisfies the Cerami we obtain a critical point $u_{2} \in H^{1}(\Omega)$ for the functional $J$ in such way that $C_{1}(J, u_{2}) \neq 0$. Here $C_{k}(J, .)$ is stand for the critical groups for $J$
at some critical point. For further results on critical groups and morse theory we infer the reader to Chang \cite{chang}.  Minimizing $J$ over the sets $A^{+}, A^{-}$ we obtain again two critical points $u_{-}, u_{+} \in H^{1}(\Omega)$ such that $C_{k}(J, u_{\pm}) = \delta_{k0} \mathbb{Z}$. As a consequence $u_{\pm}, u_{2}$  are three different critical points and problem \eqref{E1} admits at least three nontrivial solutions. So we finish the proof.
\end{proof}

\begin{proof}[Proof of Theorem \ref{th4}]
First of all, the functional $J$ satisfies the Cerami condition, see Proposition \ref{cerami}. According to Proposition \ref{auxiliar} the functional $J$ is bounded from below. Furthermore, Proposition \ref{auxiliar3} says that $J$ admits the Local Liking geometry. For the Local Linking Theorem we infer the reader to \cite{willem}. Hence using the Local Linking Theorem we obtain the existence of two nontrivial weak solutions for the problem \eqref{E1}. So we end the proof.
\end{proof}

\textbf{Acknowledgments}: The second author in this work was partially supported by CNPq/Brazil with grants 211623/2013-0.

\end{document}